\documentclass[a4paper]{article}

\RequirePackage{amsthm,amsmath,amsfonts,amssymb}
\RequirePackage[numbers]{natbib}
\usepackage{amsbsy}
\usepackage{graphicx}

\usepackage{tikz}

\newtheorem{proposition}{Proposition}
\newtheorem{theorem}{Theorem}
\newtheorem{lemma}{Lemma}
\newtheorem{corollary}{Corollary}

\def\d{\mathop{\hbox{d}}\nolimits}
\def\v{\mathop{\hbox{vol}}\nolimits}

\def\He{\mathop{\hbox{Hess}}\nolimits}

\def\C{\mathop{\mathcal C}\nolimits}
\def\H{\mathop{\mathcal H}\nolimits}

\def\N{\mathop{\mathcal N}\nolimits}
\def\Q{\mathop{\mathcal Q}\nolimits}
\def\S{\mathop{\mathcal S}\nolimits}
\def\T{\mathop{\mathcal T}\nolimits}
\def\U{\mathop{\mathcal U}\nolimits}
\def\V{\mathop{\mathcal V}\nolimits}

\def\n{\mathop{\boldsymbol n}\nolimits}
\def\M{\mathop{\boldsymbol M}\nolimits}

\DeclareMathOperator*{\argmin}{arg\,min}

\begin{document}

\title{Central limit theorem for intrinsic Fr\'echet means in smooth compact Riemannian manifolds}

\author{Thomas Hotz     \footnote{Technische Universit\"at of Ilmenau, Germany} \, ,  Huiling Le \footnote{University of Nottingham, United Kingdom} \, ,
Andrew T.A. Wood \footnote{Australian National University, Canberra, Australia} \footnote{Corresponding author}}
\date{}
\maketitle

\begin{abstract}
We prove a central limit theorem (CLT) for the Fr\'echet mean of independent and identically distributed observations in a compact Riemannian manifold assuming that the population Fr\'echet mean is unique.  Previous general CLT results in this setting have assumed that the cut locus of the Fr\'echet mean lies outside the support of the population distribution.  So far as we are aware, the  CLT in the present paper is the first which allows the cut locus to have co-dimension one or two when it is included in the support of the distribution.  A key part of the proof is establishing an asymptotic approximation for the parallel transport of a certain vector field.  Whether or not a non-standard term arises in the CLT depends on whether the co-dimension of the cut locus is one or greater than one: in the former case a non-standard term  appears but not in the latter case.  This is the first paper to give a general and explicit expression for the non-standard term which arises when the co-dimension of the cut locus is one.
\end{abstract}

\section{Introduction}

The  Fr\'echet mean,  the natural setting for which is a metric space, is defined as the point, or set of  points, in the space for which the sum of squared distances is minimised.  In Euclidean spaces and normed vector spaces, the Fr\'echet mean is the standard linear mean.  More generally, it extends the concept of the mean to nonlinear spaces.  In this paper we focus on  the  large sample behaviour of the sample Fr\'echet mean based on the intrinsic distance in smooth, compact Riemannian manifolds. 

Central limit theory for Fr\'echet means on compact Riemannian manifolds has been an ongoing topic of research for over 20 years.  The principal source of difficulty in proving a general central limit theorem for the intrinsic Fr\'echet mean is due to the so-called  cut locus of a manifold.  Roughly speaking, the cut locus of a point $x$ in a manifold $\M$ is the set of points $z \in \M$ such that there exists more than one distance-minimising geodesic from $x$ to $z$. This non-uniqueness produces non-smooth behaviour in the estimating function for the Fr\'echet mean.  However, despite the challenge posed by the cut locus, there has been some  progress in this area, typically with the limitation that the cut locus of the population Fr\'echet mean is assumed  to lie outside the support of the population distribution.

For an account of nonparametric inference for manifold-valued data see Bhattacharya and Bhattacharya \cite{BB}.  Significant contributions on central limit theorems (CLTs)  for the Fr\'echet mean in compact Riemannian manifolds include the following.  The papers of Bhattacharya and Patrangenaru \cite{BP1}, \cite{BP2} were the first to lay out an extensive Fr\'echet central limit theory for manifolds, covering both intrinsic and extrinsic means; Kendall and Le \cite{KL}  proved a CLT for Fr\'echet means based on independent but not necessarily identically distributed manifold-valued random variables;  Bhattacharya and Lin \cite{BhLi} considered a more general metric space setting than just manifolds  but also derived results of interest for manifolds; Eltzner and Huckemann \cite{EH} obtained further extensions  and they also discussed a phenomenon that they call smeariness; moreover Eltzner et al. \cite{EGHT} proved a further CLT and developed the concepts of topological stability and metric continuity of the cut locus which we make use of later in the paper.   However, all of the CLTs  for Fr\'echet means in general compact Riemannian  manifolds given in the contributions mentioned above, and to the best of our knowledge all of the the relevant literature, with the exception of Bhattacharya and Lin \cite{BhLi},  assume that the relevant population distribution has support which excludes the cut locus.

The only CLT for Fr\'echet means in general compact Riemannian manifolds in the contributions mentioned above, and to the best of our knowledge in all of the relevant literature, which does not assume that the relevant population distribution has support which excludes the cut locus is given by Theorem 3.3 in Bhattacharya and Lin \cite{BhLi}. They essentially showed that the Fr\'echet mean exhibits standard behaviour also when the cut-locus of small balls around the population mean carry mass which goes to zero faster than the radius raised to the manifold's dimension plus two. This will essentially be the case if the distribution is absolutely continuous with respect to the Riemannian volume measure and the cut-locus is of co-dimension three as is the case for three- and higher-dimensional spheres; cf. Corollary 3.5 in Bhattacharya and Lin \cite{BhLi}. In fact, the authors remark that they can treat the two-dimensional sphere only under support restrictions excluding the cut locus (see their Remark 3.7).  However, we speculate that it may be possible to use results along the lines of Brown \cite{BRO}, see also Ritov \cite{YR}, to prove a standard CLT for the Fr\'echet mean in the case of $S^2$, where the cut locus has co-dimension 2,  but we have not yet  investigated all of the details.  Here, we take a different approach to that problem.

At the outset it was not clear whether the CLT for the intrinsic Fr\'echet mean on compact Riemannian manifolds exhibits standard behaviour but with technically difficult proofs or whether non-standard behaviour can occur.  The article by Hotz and Huckemann \cite{HH}, who considered the 
intrinsic Fr\'echet mean on the circle, $S^1$,  settled the matter by showing that highly non-standard behaviour occurs in this setting.     This sets the scene for the currently open question of the appropriate form of the central limit theorem for the intrinsic Fr\'echet  mean in a general compact Riemannian manifold.   

The principal  aims of this paper are (i) to clarify when non-standard behaviour of the Fr\'echet mean in compact Riemannian manifolds occurs; and (ii) to characterise the non-standard behaviour when it does occur.   Specifically, we allow the support of the population distribution to include the cut locus and only a mild regularity assumption is made in this regard.   A key part of the proof is establishing an asymptotic approximation for the parallel transport of a certain vector field.  Whether or not a non-standard term arises in the CLT depends on whether the co-dimension of the cut locus relative to $\M$ is $1$ or greater than 1: in the former case a non-standard term will  appear but not in the latter case.  The non-standard term which arises when the co-dimension of the cut locus is 1 is precisely characterised.  

The main results of the paper, Theorem 1 and Theorem 2,  are stated in Section 2 and are proved in Section 3 and Section 4, respectively.

\section{Main Results}

\subsection{Central Limit Theorem}

Let $\M$ be a compact and connected Riemannian manifold (without boundary) of dimension $m$ and let $\rho$ denote the distance function on $\M\times\M$ induced by the Riemannian metric.  Suppose that $\mu$ is a probability measure on $M$.  The Fr\'echet function $F_\mu$ of $\mu$ is defined as
\begin{equation}
F_\mu(x)=\int_{\M}\rho(x,y)^2\d\mu(y),\qquad x\in\M.
\label{F_mu}
\end{equation}
Since $\M$ is compact, $F_\mu(x)<\infty$ for all $x\in\M$.  The population Fr\'echet mean is defined by
\[
x_0=\argmin_{x \in M}F_\mu(x).
\]
For some $\mu$, $x_0$ will consist of a subset of $\M$ rather than a single point in $\M$.   It will be  assumed throughout the paper that $x_0$ is unique.

Suppose $\xi_1, \ldots , \xi_n \in \M$ is a  random sample drawn independently from $\mu$. Then, the set of sample Fr\'echet means is defined by
\begin{equation}
\mathcal{G}_n=\argmin_{x \in \M} \sum_{i=1}^n \rho(x,\xi_i)^2,
\label{sample_mean}
\end{equation}
where $\mathcal{G}_n \subset \M$ is the set of global minima of $n^{-1} \sum_{i=1}^n \rho(\xi_i, y)^2$.  In those cases where $\mathcal{G}_n$ is not a singleton set, it is assumed that a measureable selection $\hat{\xi}_n \in \mathcal{G}_n$ has been made, so that $\hat{\xi}_n \in \mathcal{G}_n$ is a measureable random element in the case where $\mathcal{G}_n$ is not a singleton set.     

The following result, proved in Section 5.1, makes use of the strong laws of large numbers proved by Ziezold \cite{ZZ} and  Evans and Jaffe \cite{EJ}. 

\begin{proposition}
Assume that (i) $\M$ is compact and (ii) $x_0 \in \M$ is the unique population Fr\'echet mean of $\mu$.  For each $n$, let  $\hat{\xi}_n \in \mathcal{G}_n$ denote any measureable selection from $\mathcal{G}_n$.  Then $\rho(x_0, \hat{\xi}_n) \overset{a.s.} \to 0$ as $n \rightarrow \infty$.
\end{proposition}

Let $\mathcal{T}_x(\M)$ denote the tangent space at $x \in \M$ and  write $\exp_x(v)$ to denote the exponential map, which maps a point $v \in \mathcal{T}_x(\M)$ to the point $\exp_x(v) \in \M$. The inverse exponential (or log) map, denoted $\exp_x^{-1}(y)$, maps a point $y \in \M\setminus\mathcal C_x$ to the point $\exp_x^{-1}(y) \in \mathcal{T}_x(\M)$, where $\mathcal C_x$ denotes the cut locus of $x$.  See, for example, Chavel \cite{CHAV} 
for terminology.  Also, define
\begin{equation}
G_\mu(x)=\int_{\M} \exp_x^{-1}(\xi)\,1_{\{\xi\not\in\mathcal C_x\}}\d\mu(\xi), \hskip 0.2truein x \in M,
\label{G(x)}
\end{equation}
where $1_A$ denotes the indicator function of a set $A$.
Note that $\{G_\mu (x): \, x \in \M\}$ is a vector field on $\M$.   It follows from the result of \cite{LB2} that
\[G_\mu(x_0)=\int_{\M} \exp_{x_0}^{-1}(\xi)\,\d\mu(\xi)=0\in\mathcal{T}_{x_0}(\M)\]
and that, with probability one under the product measure determined by $\mu$,
\begin{equation}
G_{\hat\mu_n}(\hat\xi_n)=\frac{1}{n}\sum_{i=1}^n\exp_{\hat\xi_n}^{-1}(\xi_i)=0\in\mathcal{T}_{\hat\xi_n}(\M),
\label{stardust}
\end{equation}
where $\hat{\mu}_n$ is the empirical distribution on $\M$ based on the random sample $\xi_1, \ldots , \xi_n$.

Before stating Theorem 1 and Theorem 2, we mention a number of relevant facts.  We denote by $D$ the covariant derivative and by $\nabla$ the gradient operator, both defined on $\M$. 
For $x'\not\in\mathcal C_x$,
\begin{equation}
	\nabla\rho_x(x')^2=-2\exp_{x'}^{-1}(x)
	\label{grad_d2}
\end{equation}
(cf.  Jost \cite{JJ}, p.203).  Moreover, the Hessian, $\He^f$, of a smooth function $f$ on $\M$ is the (symmetric) $(0,2)$-tensor field such  that, for any vector fields $U$ and $V$ on $\M$,
\begin{eqnarray}
	\He^f(U,V)(x')=\langle D_V(\nabla f),U\rangle(x')
	\label{eqn1a}
\end{eqnarray}
(cf. O'Neill \cite{BO}  p.86). That is, $\hbox{Hess}^{\rho_x^2}$ can be expressed as
\begin{eqnarray}
	-\frac{1}{2}\hbox{Hess}^{\rho_x^2}(V(x'),U(x'))=\langle(H(x'\mid x))(V(x')),\,U(x')\rangle,
	\label{eqn3.1b}
\end{eqnarray}
for any smooth vector fields $U$, $V$ on $\M$ and any $x'\in\M\setminus\C_x$, where $H(x'\mid x)$ is the $(1,1)$-tensor such that, for any smooth vector field $V$ on $\M$,
\begin{eqnarray}
	(H(x'\mid x))(V(x'))=D_{V(x')}\exp^{-1}_{x'}(x).
	\label{eqn3.1a}
\end{eqnarray}

For any $x\in\M$ and $\delta>0$,  define the sets 
\begin{equation}
\mathcal{A}_\delta(x) = \bigcup_{y \in B_\delta(x)} \mathcal{C}_y
\label{curlyAdelta}
\end{equation}
and
\begin{equation}
\mathcal{B}_\delta(x)= \bigcup_{z \in \mathcal{C}_x} B_\delta(z) =\{y \in \M: \rho(z,y)<\delta \hskip 0.1truein \hbox{for some} \hskip 0.1truein  z \in \mathcal{C}_x\},
\label{curlyBdelta}
\end{equation}
where $B_\delta (x) = \{x' \in \M:\, \rho(x,x') < \delta\}$.  The concepts of topological stability and metrical continuity of the cut locus are relevant in the present context; see definitions 3.6 and 3.10 in Eltzner et al. \cite{EGHT}.   Corollary 3.8 and Proposition 3.11 in Eltzner et al. \cite{EGHT} prove that both topological stability and metric continuity of the cut locus hold for compact Riemannian manifolds. Here, it will be slightly more convenient to use the concept of metric continuity at any point $x \in \M$.  In the notation defined above, metric continuity entails the following.
\begin{proposition}
If $\M$ is a compact Riemannian manifold and  $x \in \M$, 	then for any $\delta>0$ there exists a $\delta_1>0$ such that
\begin{equation}
	\mathcal{A}_{\delta_1}(x) \subseteq \mathcal{B}_\delta(x).
	\label{metric_stability}
\end{equation}
\end{proposition}
In Appendix A we give a different proof for Proposition 2 to that given by Eltzner et al. \cite{EGHT}.

Let $\textrm{vol}_{\M}$ denote the Riemannian volume measure on $\M$.   The key  linearization result we need is the the following.

\begin{theorem}
Assume that (i) $\M$ is a compact, connected  Riemannian manifold; (ii) $x_0 \in \M$ is the unique population Fr\'echet mean of $\mu$; (iii) for $\delta>0$ sufficiently small, $\mu$, restricted to $\mathcal{B}_\delta(x_0)$ defined in $(\ref{curlyBdelta})$, is absolutely continuous with respect to $\textrm{vol}_{\M}$ and the corresponding Radon-Nicodym derivative has a version $\psi$ which is continuous on $\mathcal{B}_\delta(x_0)$ ; (iv) as $\delta \downarrow 0$,  $\mathcal{A}_\delta(x_0)$ in $(\ref{curlyAdelta})$ satisfies $\textrm{vol}_{\M}(\mathcal{A}_\delta(x_0))= O(\delta)$; (v)  the integral
\begin{equation}
	\int_{\M} H(x_0\vert \xi) \d \mu(\xi)
	\label{Hessian21}
\end{equation}
exists, where $H(\cdot\vert \cdot)$ is defined in (\ref{eqn3.1a}).
 Then the vector field $G_\mu(x)$ admits the following linearization for $x \in \M$ in a neighbourhood of $x_0$:
\begin{equation}
\Pi_{x,x_0} G_\mu (x)=-\Psi_{\mu}(x_0) \exp_{x_0}^{-1}(x)+ R(x,x_0),
\label{psi}
\end{equation}
where $\Pi_{x,x_0}$ denotes parallel transport from $\T_x(\M)$ to $\T_{x_0}(\M)$ along the (unique) shortest geodesic between $x$ and $x_0$, $\vert \vert R(x,x_0)\vert \vert = o(\rho(x,x_0))$ and $\Psi_{\mu}(x_0)$ is a $(1,1)$-tensor defined below in (\ref{eqn5c}).
\end{theorem}

The proof of Theorem 1,  given in Section 3, uses some involved geometric arguments.  These arguments are of potentially broader interest than just the current context.   The definition of $\Psi_{\mu}(x_0)$, which is given in the next subsection,  has a particularly interesting form when the co-dimension of the cut-locus of $x_0$ is 1.  In this case $\Psi_{\mu}(x_0)$ contains a non-standard term which we discuss in detail below, and illustrate in some examples at the end of the section.  

Note that, under assumption (iii) of Theorem 1, $G_\mu$ defined by \eqref{G(x)} can be written as
\[G_\mu(x)=\int_{\M}\exp^{-1}_x(\xi)\,\d\mu(\xi)\] 
in a neighbourhood of $x_0$ and so we have the relationship 
\[G_\mu(x)=-2\nabla F_\mu(x)\]
in that neighbourhood. Then, one immediate consequence of Theorem 1 is that the Hessian of the Fr\'echet function $F_\mu$ at $x_0$ exists and it can be expressed in terms of $\Psi_\mu$ as 
\begin{equation}
	\hbox{Hess}^{F_\mu}(U,V)(x_0)=2\langle\Psi_\mu(x_0)(U(x_0),\,V(x_0)\rangle.
\label{Hess199}
\end{equation}
In fact, a slight modification of the proof for Theorem 1 shows that the same result holds in a neighbourhood of $x_0$.

\vskip 6pt
We now state our main result, a CLT for $\hat{\xi}_n$, assumed to be a measurable selection from $\mathcal{G}_n$. 

\begin{theorem}
Suppose that assumptions (i) -- (v) of Theorem $1$ hold and that $\hat{\xi}_n$ is any measurable selection from $\mathcal{G}_n$, as in Proposition $1$.  In addition, assume (vi)  that $\Psi_\mu(x_0)$  is strictly positive definite. Then
\[
\sqrt{n} \exp_{x_0}^{-1}(\hat{\xi}_n) \overset{d} \to \mathfrak{N}_m \left \{0_m, \Psi_{\mu}(x_0)^{-1}V_0 (\Psi_{\mu}(x_0)^\top)^{-1}\right \},
\]
where $V_0=\hbox{\rm Cov}(\exp^{-1}_{x_0}(\xi_1))$.
\end{theorem}

\subsection{Discussion of assumptions}

Here we discuss the assumptions made in Theorem 1 and Theorem 2.  Assumptions (i) and (ii) in Theorem 1 define the setting that we consider.  Assumption (iii) in Theorem 1 implies a certain level of regularity of the population distribution in a neighbourhood of the cut locus of the Fr\'echet mean; some such regularity is needed for an expansion of the type (\ref{psi}) to hold.  Previous central limit theorems in this setting, such as Bhattacharya and Patragenaru \cite{BP1}, \cite{BP2} have made the much stronger assumption that the population probability density function is zero in a neighbourhood of the cut locus of the population  Fr\'echet mean.
Bhattacharya and Lin \cite{BhLi} have assumed $\mu( \mathcal A_\delta(x_0)) = o(\delta^2)$ whereas our assumptions (iii) and (iv) amount only to $\mu( \mathcal A_\delta(x_0)) = O(\delta)$.

Assumptions (iv) and (v) in Theorem 1 are largely geometric in character.  For each of these  assumptions, it would be interesting to know whether or not it holds for all smooth, compact connected manifolds when the population Fr\'echet mean is unique.  However, we do not have a proof or a counter-example to this statement in either case and we have found nothing in the literature that throws light on either question.  

Finally, assumption (vi) in Theorem 2 is a non-degeneracy assumption.  
If $\Psi_\mu(x_0)$ is non-negative definite but not of full rank then we are in the same situation as that of a smeary central limit theorem, as discussed by Eltzner and Huckemann \cite{EH}: specifically,  a central limit theorem  is expected to hold but with a non-standard convergence rate which depends on the level of smoothness of the population distribution.  Bearing in mind that $2 \Psi_\mu(x_0)$ is the Hessian of the Fr\'echet function $F_\mu(x)$, see (\ref{Hess199}), it follows that if $\Psi_\mu(x_0)$ has one or more strictly negative eigenvalues then this contradicts $x_0$ being a Fr\'echet mean due to the Hessian of the Fr\'echet function $F_\mu(x)$ in (\ref{F_mu}) not being non-negative definite, in which case $x_0$ can not be a stationary minimum of the Fr\'echet function.

\subsection{The expression of $\Psi_{\mu}(x_0)$}

The expression of $\Psi_\mu(x_0)$ comprises two terms, one associated with the Hessian of the squared distance function, away from the cut locus of $x_0$, and the other with the behaviour of the distance function on the cut locus $\mathcal C_{x_0}$ of $x_0$. Hence, 
the second term reflects the geometric structure of the manifold $\M$. 

\vskip 6pt
To make the notation more explicit, we write, for any fixed $x\in\M$, $\rho_x=\rho(x,\,\cdot\,)$. Note that $\rho_x^2$ is a smooth function away from the cut locus of $x$. 
The tensor $H(x_0\mid\cdot\,)$ which appears in (\ref{eqn3.1b}) and (\ref{eqn3.1a}) determines the first term of $\Psi_\mu(x_0)$. The construction above for $H(x_0\mid x)$ requires that $x\not\in\C_{x_0}$. Nevertheless, it follows from the result of Le and Barden \cite{LB2}  that $H(x_0\mid\xi_1)$ is well-defined with probability one, because condition (iii) of Theorem 1 implies that $\mu(\xi \in \mathcal{C}_{x_0})=0$, i.e. the cut locus of $x_0$ has zero probability under $\mu$.

\vskip 6pt
To introduce the second term of $\Psi_\mu(x_0)$, we first recall some facts on the cut locus $\mathcal C_x$ of $x$ and the behaviour of $\rho_x$ nearby. These results, explicitly or implicitly stated in Barden and Le \cite{BL1} \& Le and Barden  \cite{LB}, are given in the following lemmas. The first one is on the structure of $\mathcal C_x$, a set of co-Hausdorff-dimension at least one (or, equivalently, where the Hausdorff dimension is at most $m-1$). 

\begin{lemma}
For any $x \in \M$ there is a set $\Q_x$ of Hausdorff $(m-1)$-measure zero contained in $\C_x$ and \textit{containing} the first conjugate locus of $x$ such that $\H_x=\C_x\setminus\Q_x$ is a countable union of disjoint hyper-surfaces (co-dimension one sub-manifolds) where, for each $y\in\H_x$, there are exactly two minimal geodesics from $x$ to $y$. In particular, $\H_x$ is a Borel measurable set and $y\in\H_x$ if and only if $x\in\H_y$.
\label{le3}
\end{lemma}

The decomposition of $\C_x$ in Lemma \ref{le3} above is the same as that given in Theorem 2 of \cite{LB}, but slightly different from that given in Prop 2 in \cite{BL1}. In \cite{BL1} $\Q_x$ is the set of the first conjugate loci of $x$ in $\C_x$, while here $\Q_x$ is the union of the set of the first conjugate loci of $x$ in $\C_x$ with the set of non-conjugate points in $\C_x$ which have more than two minimal geodesics to $x$. Furthermore, the proof of Theorem 2 in Le and Barden \cite{LB} made it clear that the set $\Q_x$, which was called $E$ there, has co-dimension at least two, although the Theorem itself only stated that it has Hausdorff $(m-1)$-measure zero as needed for that paper. In particular, that the set of the first conjugate loci of $x$ has co-dimension at least two was proved in Proposition 1 of Barden and Le \cite{BL1}.  

The next two lemmas show that, although $\rho_x$ is not differentiable at $\mathcal C_x$, it is relatively well behaved in a neighbourhood of $\mathcal H_x$.

\begin{lemma}
Let $\mathcal H_x$ be given as in Lemma $\ref{le3}$. For each $y\in\H_x$, there is a neighbourhood $\V_y$ of $y$ in $\M$ on which there are two unique smooth functions $\phi_{1y}(\,\cdot\mid x)$ and $\phi_{2y}(\,\cdot\mid x)$ such that for any $y'\in\V_y$,
\[\rho_x(y')=\min\{\phi_{1y}(y'\mid x),\phi_{2y}(y'\mid x)\},\]
where $\phi_{1y}(y'\mid x)=\phi_{2y}(y'\mid x)$ if and only if $y'\in\V_y\bigcap\H_x$. 
\label{le4}
\end{lemma}

The neighbourhood $\V_y$ and the two functions $\phi_{iy}(\,\cdot\mid x)$ in the above Lemma were constructed in the proof of Proposition 1 in Barden and Le \cite{BL1} as follows. There are two \textit{disjoint} neighbourhoods $\U_{1y}$ and $\U_{2y}$ in $\T_x(\M)$ such that, for each $i$, $\V_y=\exp_x(\U_{iy})$. Then, $\phi_{iy}(y'\mid x)=\|(\exp_x^{-1}\mid_{\U_{iy}}(y')\|$ for $y'\in\V_y$.

The next result is an immediate consequence of this construction. 

\begin{lemma}
Let $\mathcal H_x$ be given as in Lemma $\ref{le3}$, and let $\mathcal V_y$, $\mathcal U_{jy}$ and $\phi_{jy}$ be given as in Lemma $\ref{le4}$ and in the following construction. If, for $y'\in\V_y\bigcap\H_x$, $\gamma_j$ is the minimal geodesic with $\gamma_j(0)=x$, $\gamma_j(1)=y'$ and $\dot\gamma_j(0)\in\U_{jy}$, then 
\[\nabla\phi_{jy}(y'\mid x)=\dot\gamma_j(1)/\|\dot\gamma_j(1)\|,\]
that is, $\nabla\phi_{jy}(y'\mid x)$ is the unit tangent vector to $\gamma_j$ at $y'$.
\label{le5}
\end{lemma}

The following result follows from the uniqueness of the pair of functions $\phi_{jy}$, $j=1,2$,  stated in Lemma \ref{le4}.  

\begin{corollary}
Let $\mathcal H_x$ be given as in Lemma $\ref{le3}$, and let $\mathcal V_y$ and $\phi_{jy}$ be given as in Lemma $\ref{le4}$. For each $y'\in\V_y\bigcap\H_x$, the unordered pair of the functions $\{\phi_{1y'}(\cdot\mid x),\phi_{2y'}(\cdot\mid x)\}$ coincides with the pair $\{\phi_{1y}(\cdot\mid x),\phi_{2y}(\cdot\mid x)\}$ on $\V_y\bigcap\V_{y'}$. Thus, the difference $\phi_{1z}(\cdot\mid x)-\phi_{2z}(\cdot\mid x)$ is, up to sign, independent of $z\in(\V_y\cup\V_{y'})\cap\H_x$ and so, making a continuous choice of sign, this difference is a well-defined function $\chi_i(\cdot\mid x)$ on a neighbourhood (in $\M$) of each connected component $\H_i(x)$ of $\H_x$.
\label{co1}
\end{corollary}

This, together with the results in Barden and Le \cite{BL1}, implies the following relationship between $\mathcal H_i$ and $\chi_i$.

\begin{corollary}
Let $\mathcal H_i$ and $\chi_i$ be given as in Corollary $\ref{co1}$. For $y\in\H_i(x)$, $\nabla\chi_i(y\mid x)$ is non-zero and normal to $\H_i(x)$ at $y$.
\label{co1a}
\end{corollary}

With the above understanding of $\mathcal C_x$ and $\rho_x$ nearby, we reach the following main ingredients for our definition of the second term of $\Psi_\mu(x_0)$.

\begin{corollary}
Let $\mathcal H_x$ be given as in Lemma $\ref{le3}$, and let $\mathcal H_i$ and $\chi_i$ be given as in Corollary $\ref{co1}$. Then,
\begin{enumerate}
\item[$(a)$] the set $\H_x$ can be expressed as the countable union of disjoint $\H_i(x)$;
\item[$(b)$] the function
\begin{eqnarray}
\kappa(y\mid x)=\left\|\nabla\chi_i(y\mid x)\right\|,\qquad\hbox{ if }y\in\H_i(x),
\label{eqn1b}
\end{eqnarray}
is well-defined on $\H_x$;  
\item[$(c)$] the unit normal vector field given by
\begin{eqnarray}
\n(y\mid x)=\frac{\nabla\chi_i(y\mid x)}{\kappa(y\mid x)}\in\T_y(\M),\qquad\hbox{ if }y\in\H_i(x),
\label{eqn1c}
\end{eqnarray}
is well-defined up to sign on $\H_x$. 
\end{enumerate}
\end{corollary}

Note that, for $y\in\H_i(x)$,
\[\rho(x,y)\,\kappa(y\mid x)=\frac{1}{2}\left\|\nabla\left(\phi_{1y}(y\mid x)^2-\phi_{2y}(y\mid x)^2\right)\right\|;\]
and that, for $y\in\H_i(x)$ and $y'\in\V_y\bigcap\H_i(x)$,
\[\nabla\phi_{1y}(y'\mid x)=\lim_{z\rightarrow y',\langle\exp^{-1}_{y'}(z),\nabla\phi_i(y'\mid x)\rangle<0}\nabla\phi_{1y}(z\mid x).\]

\vskip 6pt
Now, for $y\in\H_x$, define $\d y^{\perp_x}$ to be the 1-form, unique up to sign, given by $\d y^{\perp_x}(U(y))=\langle\n(y\mid x),U(y)\rangle$ for any tangent vector $U(y)$ at $y$. 
Write $J(y\,|\, x)$ for the well-defined $(0,2)$-tensor at $y$ on $\H_x$ given by
\begin{eqnarray}
J(y\mid x)=\rho_x(y)\,\kappa(y\mid x)\,{\rm d}y^{\perp_x}\otimes{\rm d}y^{\perp_x}.
\label{eqn2}
\end{eqnarray}
That is, for any $y\in\H_x$ and any $U(y),V(y)\in\mathcal{T}_y(\M)$,
\[(J(y\mid x))(U(y),V(y))=\rho_x(y)\,\kappa(y\mid x)\,\langle\n(y\mid x),U(y)\rangle\,\langle\n(y\mid x),V(y)\rangle.\]

Write $\alpha(t)$ for the unit speed geodesic orthogonal to $\H_x$ at $y$ and $\tau_{x|y}(t)$ for the distance from $x$ to $\H_{\alpha(t)}$ along the geodesic orthogonal to $\H_y$. Then
\begin{eqnarray}
\d x^{\perp_y}=\tau'_{x|y}(0)\,\d y^{\perp_x}.
\label{eqn5b}
\end{eqnarray}
That is, $\tau'_{x|y}(0)$ represents the rate of change of $x$ orthogonal to $\mathcal H_y$ as $y$ moves orthogonally to $\mathcal H_x$.
In terms of $J(x_0\mid y)$, $\tau'_{x|y}(0)$ and $\psi$, the Radon-Nikodym derivative of $\mu$ with respect to the volume measure in a neighbourhood of $\mathcal C_{x_0}$, we denote by $J_\mu(x_0)$ the $(1,1)$-tensor defined by
\begin{eqnarray}
\begin{array}{rcl}
&&\hspace*{-5mm}\left(J_\mu(x_0)\right)(V(x_0))\\
&=&\displaystyle\int_{\H_{x_0}}\!\!\!\!\!\rho_y(x_0)\,\kappa(x_0\!\mid\! y)\,
\langle\n(x_0\!\mid\! y),V(x_0)\rangle\,\n(x_0\!\mid\! y)\,\tau'_{y|x_0}(0)\, \psi(y)\,\d\v_{\H_{x_0}}(y),
\end{array}
\label{eqn7}
\end{eqnarray}
where $\d\v_{\H_x}$ denotes the co-dimension one surface measure on $\H_x$.

\vskip 6pt
Finally, we can express the $(1,1)$-tensor $\Psi_\mu(x_0)$ appearing in Theorem 1 and Theorem 2 explicitly.  
\begin{lemma}
	In the notation introduced above, 
\begin{eqnarray}
\Psi_\mu(x_0)=-\int_{\M}H(x_0\mid\xi)\,\d\mu(\xi)-J_\mu(x_0).
\label{eqn5c}
\end{eqnarray}
\end{lemma}

\subsection{Three examples of $J_\mu$}

In the case of symmetric spaces, $\tau'_{y|x}(0)\equiv1$ and so the expression for $J_\mu(x)$ defined by \eqref{eqn7} can be simplified. We now calculate $J_\mu(x)$  for special symmetric spaces with appropriate `coordinate systems'.  Moreover, we show that condition (iv) in Theorem 1 is  satisfied in each of the three examples.  i.e. we show that $\textrm{vol}_{\M}(\mathcal{A}_\delta(x))=O(\delta)$ as $\delta \downarrow 0$, where $\mathcal{A}_\delta(x)$ is defined in (\ref{curlyAdelta}); and we show that condition (v) in Theorem 1 is also satisfied in the three examples.

\vskip 6pt
$(a)$ $\M=S^1$: $\H_x=\C_x$ contains only the antipodal point $y$ of $x$. Thus, $\rho(x,y)=\pi$; the initial tangent vectors of the two geodesics from $x$ to $y$ have the opposite direction so that $\kappa(x\mid y)=2$; and we may take $\n(x\mid y)=1$. Hence, if we take the standard coordinate in the subset $(-\pi,\pi]$ in its universal cover with $x=0$, then the corresponding $J_\mu$ is $J_\mu(0)=-2\pi\,\psi(\pi)$, identical with the extra term in the covariance of the central limit theorem of Hotz and Huckemann \cite{HH}.

Finally, we check conditions (iv) and (v) of Theorem 1. %and condition (vi) of Theorem 2.  
Since $\mathcal{C}_x$ is the antipodal point of $x$, it follows that, in the local coordinates introduced above,  $\mathcal{A}_\delta$ may be written as $\mathcal{A}_\delta(0)=(-\pi, -\pi +\delta) \cup (\pi-\delta, \pi]$, so that $\textrm{vol}_{\M}(\mathcal{A}_\delta(x))=2\delta$ and 
therefore condition (iv) %and (v) 
of Theorem 1 is %are 
satisfied. Condition (v) 
%of Theorem 2 
follows because the circle is flat and therefore the Hessian $H(x_0\vert \xi)=1$ if $\xi$ is not the antipodal point of $x_0$. 

For higher dimensional spheres $S^d$, $d > 1$, we have $\textrm{vol}_{\M}(\mathcal A_\delta(x)) = O(\delta^d)$ but $\H_x$ is empty since the cut-locus is of co-dimension $d > 1$, so $J_\mu(0)$ vanishes. For $d > 2$ this has already been observed by Bhattacharya \& Lin \cite{BhLi} but the CLT for $S^2$ given a non-vanishing density at the cut locus appears to be new.

\vskip 6pt
$(b)$ $\M=S^1\times S^1$ (the standard torus): We take the standard coordinate system in the subset $(-\pi,\pi]\times(-\pi,\pi]$ in its universal covering space with $x=(0,0)$. Then $\C_x=\H_x\cup\{(\pi,\pi)\}$ where $\H_x$ is the union of two disjoint sets $\H_1(x)$ and $\H_2(x)$ where $\H_1(x)=(-\pi,\pi)\times\{\pi\}$ and $\H_2(x)=\{\pi\}\times(-\pi,\pi)$. Under this coordinate system, $U=\frac{\partial}{\partial x_1}$ and $V=\frac{\partial}{\partial x_2}$ form an orthonormal basis of $\T_x(\M)$, and, for any $y=(y_1,y_2)\in\M$, $\rho(x,y)^2=y_1^2+y_2^2$. Also, up to sign, for $y\in\H_1(x)$, $\n(x\mid y)=V$ and, for $y\in\H_2(x)$, $\n(x\mid y)=U$. For $y\in\H_1(x)$, $\rho(x,y)\kappa(x\mid y)=2|y_1|=2\pi$ and, similarly, $\rho(x,y)\kappa(x\mid y)=2|y_2|=2\pi$ for $y\in\H_2(x)$. Thus,
\begin{eqnarray*}
&&\hspace*{-5mm}\int_{\H_x}\rho(x,y)\,\kappa(x\mid y)\,\langle\n(x\mid y),U\rangle\,\langle\n(x\mid y),U\rangle\,\psi(y)\,\d\v_{\H_x}(y)\\
&=&\int_{\H_2(x)}\rho(x,y)\,\kappa(x\mid y)\,\langle\n(x\mid y),U\rangle\,\langle\n(x\mid y),U\rangle\,\psi(y)\,\d\v_{\H_2(x)}(y)\\
&=&2\pi\int_{-\pi}^\pi\psi(y_1,\pi)\,\d y_1;\\
&&\hspace*{-5mm}\int_{\H_x}\rho(x,y)\,\kappa(x\mid y)\,\langle\n(x\mid y),V\rangle\,\langle\n(x\mid y),V\rangle\,\psi(y)\,\d\v_{\H_x}(y)\\
&=&\int_{\H_1(x)}\rho(x,y)\,\kappa(x\mid y)\,\langle\n(x\mid y),V\rangle\,\langle\n(x\mid y),V\rangle\,\psi(y)\,\d\v_{\H_1(x)}(y)\\
&=&2\pi\int_{-\pi}^\pi\psi(\pi,y_2)\,\d y_2;\\
&&\hspace*{-5mm}\int_{\H_x}\rho(x,y)\,\kappa(x\mid y)\,\langle\n(x\mid y),U\rangle\,\langle\n(x\mid y),V\rangle\,\psi(y)\,\d\v_{\H_x}(y)\\
&=&0.
\end{eqnarray*}
Hence, in this case, under the chosen `coordinate system', the corresponding $J_\mu$ is
\[J_\mu(0,0)=-2\pi\begin{pmatrix}\displaystyle\int_{-\pi}^\pi\psi(y_1,\pi)\,\d y_1&0\\0&\displaystyle\int_{-\pi}^\pi\psi(\pi,y_2)\,\d y_2\end{pmatrix}.\]

Finally, we note that 
\begin{equation*}
\mathcal{A}_\delta(x) = \left \{S^1 \times \left  \{ (- \pi+\delta, -\pi) \cup (\pi-\delta, \pi]\right \} \right\} \cup \left \{ \left \{ (- \pi+\delta, -\pi) \cup (\pi-\delta, \pi] \right \} \times S^1 \right \}\!.
\end{equation*}
In this case,  $\textrm{vol}_{\M}(\mathcal{A}_\delta(x))$
is seen to be bounded by $8 \pi \delta$. It follows that   condition (iv)  of Theorem 1 is satisfied.  Condition (v) of Theorem 1 follows because the torus is flat and hence $H(x_0\vert \xi)=I$, the identity, for $\xi\not\in\C_{x_0}$.

For $d$-dimensional tori with $d > 2$, $\H_x$ is given by the union of $(d-1)$-dimensional tori, and the conditions remain satisfied with $J_\mu$ not vanishing in general.

\vskip 6pt
$(c)$ $\M=\mathbb{RP}^2$ (two-dimensional real projective space): $\Q_x=\emptyset$ so that $\H_x=\C_x$; and for any $y\in \C_x$, $\rho(x,y)=\pi/2$ where the initial tangent vectors of the two minimal geodesics from $x$ to $y$ are in opposite directions. Hence, for $y\in \C_x$, $\rho(x,y)\,\kappa(x\mid y)=\pi$. We take the normal coordinates centred at $x$ on $\T_x(\M)$. Then, using the corresponding polar coordinates $(r,\theta)$, for any $y\in\C_x$ one of the initial unit tangent vectors to the two geodesics from $x$ to $y$ has coordinates $(\cos\theta,\sin\theta)$ where $\theta\in[0,\pi)$, which we take as $\n(x\mid y)$. Thus, for $y\in\C_x$,
\[\n(x\mid y)\otimes\n(x\mid y)=\begin{pmatrix}\cos\theta\\ \sin\theta\end{pmatrix}\begin{pmatrix}\cos\theta&\sin\theta\end{pmatrix}=\begin{pmatrix}\cos^2\theta&\sin\theta\cos\theta\\ \sin\theta\cos\theta&\sin^2\theta\end{pmatrix}\]
so that in this case, under this coordinate system, the corresponding $J_\mu$ is
\begin{eqnarray*}
J_\mu(0,0)=-\pi\int_0^\pi\begin{pmatrix}\cos^2\theta&\sin\theta\cos\theta\\ \sin\theta\cos\theta&\sin^2\theta\end{pmatrix}\psi(\pi/2,\theta)\,\d\theta.
\end{eqnarray*}
This expression can be verified by direct computation of the Hessian of $F_\mu$.

Finally, we consider conditions (iv) and (v) of Theorem 1.  We first identify the form of $\mathcal{A}_\delta(x_0)$.  Without loss of generality we take $x_0$ to be $x_0= (0,0,1)^\top$ and represent $\mathbb{RP}^2$ by the hemisphere $\{x =(x_1, x_2, x_3)^\top \in \mathcal{S}^2:\,  x_3 \geq 0\}$.  Then it is easy to see that $\mathcal{A}_\delta$ is given by
\[
\mathcal{A}_\delta = \{ (\sin \theta \cos \phi, \sin \theta \sin \phi, \cos \theta)^\top: \, \theta \in ((\pi/2) - \delta, \pi/2), \phi \in (0,2\pi)\}.
\]
Moreover, the volume of $\mathcal{A}_\delta$ with respect to surface area measure on $S^2$ is $2 \pi \sin \delta$.  It follows easily that condition (iv) of Theorem 1 is satisfied here, too.

Condition (v) requires a bit more work to check in this example.  From Kendall and Le (2011), $H(x\vert y)$ on the sphere $S^2$ is given by the map
\begin{align*}
H(x\mid y):\,\, v\,\mapsto\,&\dfrac{1-\rho(x,y)\cos(\rho(x,y))/\sin(\rho(x,y))}{\rho(x,y)^2}\langle\exp^{-1}_x(y),v\rangle\exp^{-1}_x(y)\\
%&\hskip 1.0truein 
&+\dfrac{\rho(x,y)\cos(\rho(x,y))}{\sin(\rho(x,y))}v,
\end{align*}
where $\langle \cdot , \cdot \rangle$ is the Riemannian inner product on the tangent space at $x \in \M$.
When restricted to the (open) half sphere centred at $y$, it gives $H(x\mid y)$ on $\mathbb{RP}^2$. For given $x \in \M$ there is a possible singularity at $y=x$. However, the singularity is in fact a removable singularity because, for $x$ close to $y$, $\rho(x,y) \sim \sin (\rho(x,y))$ and $1- \cos(\rho(x,y)) \sim \rho(x,y)^2$.
Then, the boundedness of $H(x\mid y)$ ensures that
\[
\int_{\M} H(x\vert y)\d \mu(y)
\] 
is well-defined.

As is the case with higher-dimensional tori, it is easy to see that also for $\mathbb{RP}^d$ with $d > 2$,  conditions (iv) and (v) remain satisfied but that $J_\mu(x_0)$ will not vanish in general. To the best of our knowledge, the corresponding CLTs are the first of their kind when the cut locus is containted in the support of the distribution.

\section{Proof of Theorem 1}

To prove Theorem 1, we first consider a generalised version of the Taylor expansion of the inverse exponential map at different base points. That is, for fixed $z\in\M$, we study the Taylor expansion for the vector field $\exp^{-1}_x(z)$ for $x\not\in\C_z$. For this, we fix $z\in\M$ and, for $x_0,x_1\not\in\C_z$ sufficiently close, denote by $\gamma$ the unit speed geodesic segment such that $\gamma(0)=x_0$ and $\gamma(\rho(x_0,x_1))=x_1$.

If $\gamma(t)\not\in\C_z$ for all $t\in(0,\rho(x_0,x_1))$, $\exp^{-1}_{\gamma(t)}(z)$ is a smooth vector field along $\gamma$. Then, it follows from the definition of the covariant derivative that the Taylor expansion for $\exp_x^{-1}(z)$ about $\exp^{-1}_{x_0}(z)$ takes the form
\begin{eqnarray}
\Pi_{x_1,x_0}(\exp^{-1}_{x_1}(z))=\exp^{-1}_{x_0}(z)+(H(x_0\mid z))(\exp_{x_0}^{-1}(x_1))+R(x_0,x_1),
\label{eqn3.1}
\end{eqnarray}
where $H(x'\mid x)$ is defined by \eqref{eqn3.1a} for $x'\in\M\setminus\C_x$ and $\vert \vert R(x_0,x_1) \vert \vert = o(\rho(x_0,x_1))$. .

\vskip 6pt
In the case that $\gamma(t)\in\mathcal H_z\subseteq\mathcal C_z$ for some $t\in(0,\rho(x_0,x_1))$, we have the following result on the approximation of $\Pi_{x_1,x_0}\left(\exp^{-1}_{x_1}(z)\right)$ in terms of $\exp^{-1}_{x_0}(z)$, generalising the Taylor expansion \eqref{eqn3.1} for smooth vector fields.

\begin{proposition}
Let $x_0,x_1,z\in\M$ be such that $x_0,x_1\not\in\C_z$ are sufficiently close and $\gamma$ be the minimal unit speed geodesic from $x_0$ to $x_1$. If there is a parameter $t_z\in(0,\rho(x_0,x_1))$ such that $\gamma(t_z)\in\H_z$, then
\begin{eqnarray}
\begin{array}{rcl}
% &&\hspace*{-5mm}
~\hspace*{6mm}\Pi_{x_1,x_0}\left(\exp^{-1}_{x_1}(z)\right)
&\!=&\exp^{-1}_{x_0}(z)+(H(x_0\mid z))\left(\exp^{-1}_{x_0}(x_1)\right)\\
&&+\,\,\rho_z(\gamma(t_z))\,\kappa(\gamma(t_z)\mid z)\,\Pi_{\gamma(t_z),x_0}\left(\n(\gamma(t_z)\mid z)\right)+o(\rho(x_0,x_1)),
\end{array}
\label{eqn3.2}
\end{eqnarray}
where $H(x'\mid x)$ is defined by \eqref{eqn3.1a}, $\kappa(y\mid x)$ is defined by \eqref{eqn1b} and $\n(y\mid x)$ is defined by \eqref{eqn1c}.
\label{prop2}
\end{proposition}

\begin{proof}
If there is a parameter $t_z\in(0,\rho(x_0,x_1))$ such that $\gamma(t_z)\in\H_z\subseteq\C_z$, such $t_z$ is unique provided $x_0$ and $x_1$ are sufficiently close. Without loss of generality, we may assume that the two smooth functions $\phi_i(\,\cdot\,)=\phi_{i\gamma(t_z)}(\,\cdot\mid z)$, where $\phi_{iy}(\,\cdot\mid x)$ are defined in Lemma \ref{le4}, are chosen such that
\[\rho_z(\gamma(t))=\begin{cases}\phi_1(\gamma(t))&\hbox{if }0\leqslant t\leqslant t_z\\ \phi_2(\gamma(t))&\hbox{if }t_z\leqslant t\leqslant\rho(x_0,x_1).\end{cases}\]
Then, the difference between the two tangent vectors $\Pi_{x_1,x_0}\left(\exp^{-1}_{x_1}(z)\right)$ and $\exp^{-1}_{x_0}(z)$, both in $\mathcal{T}_{x_0}(\M)$, can be expressed as
\begin{eqnarray*}
&&\hspace*{-5mm}\Pi_{x_1,x_0}\left(\exp^{-1}_{x_1}(z)\right)-\exp^{-1}_{x_0}(z)\\
&=&\left\{\Pi_{x_1,x_0}\left(\exp^{-1}_{x_1}(z)\right)-\Pi_{\gamma(t_z),x_0}\left(-\nabla\phi_2(\gamma(t_z))\,\rho_z(\gamma(t_z))\right)\right\}\\
&&+\left\{\Pi_{\gamma(t_z),x_0}\left(-\nabla\phi_1(\gamma(t_z))\,\rho_z(\gamma(t_z))\right)-\exp^{-1}_{x_0}(z)\right\}\\
&&+\,\,\left\{\Pi_{\gamma(t_z),x_0}\left(\nabla\phi_1(\gamma(t_z))\,\rho_z(\gamma(t_z))-\nabla\phi_2(\gamma(t_z))\,\rho_z(\gamma(t_z))\right)\right\}\!.
\end{eqnarray*}
The definitions for $\kappa(y\mid x)$ and $\n(y\mid x)$ given respectively by \eqref{eqn1b} and \eqref{eqn1c} imply that the terms in the third curly bracket on the right hand side above is equal to
\[\rho_z(\gamma(t_z))\,\kappa(\gamma(t_z)\mid z)\,\Pi_{\gamma(t_z),x_0}(\n(\gamma(t_z)\mid z)).\]
By \eqref{eqn3.1}, the difference between the terms in the second curly bracket on the right hand side above and $(H(x_0\mid z))(\exp^{-1}_{x_0}(\gamma(t_z)))$ is $o(\rho(x_0,x_1))$. Since
\[\Pi_{x_1,x_0}^{-1}\circ\Pi_{\gamma(t_z),x_0}=\Pi_{x_0,x_1}\circ \Pi_{\gamma(t_z),x_0}=\Pi_{\gamma(t_z),x_1},\]
a similar application of \eqref{eqn3.1} to the terms in the first curly bracket results in
\begin{equation*}
-\Pi_{x_1,x_0}\left((H(x_1\mid z))(\exp^{-1}_{x_1}(\gamma(t_z)))\right)\!,
\end{equation*}
up to a term of order $o(\rho(x_0,x_1))$. Hence,
\begin{eqnarray*}
&&\hspace*{-5mm}\Pi_{x_1,x_0}\left(\exp^{-1}_{x_1}(z)\right)\\
&=&\exp^{-1}_{x_0}(z)+(H(x_0\mid z))\left(\exp^{-1}_{x_0}(\gamma(t_z))\right)\\&&+\,\,\Pi_{x_1,x_0}\left((H(x_1\mid z))\left(-\exp^{-1}_{x_1}(\gamma(t_z))\right)\right)\\
&&+\,\,\rho_z(\gamma(t_z))\,\kappa(\gamma(t_z)\mid z)\,\Pi_{\gamma(t_z),x_0}\left(\n(\gamma(t_z)\mid z)\right)+o(\rho(x_0,x_1)).
\end{eqnarray*}
However, using
\[\exp^{-1}_{x_0}(x_1)=\frac{\rho(x_0,x_1)}{\rho(x_0,\gamma(t_z))}\exp^{-1}_{x_0}\left(\gamma(t_z)\right)\]
and similarly for $\exp^{-1}_{x_1}\left(\gamma(t_z)\right)$, as well as noting $\rho(x_0,x_1)=\rho(x_0,\gamma(t_z))+\rho(\gamma(t_z),x_1)$, we have
\begin{eqnarray*}
&&\hspace*{-5mm}(H(x_0\mid z))\left(\exp^{-1}_{x_0}(\gamma(t_z))\right)+\Pi_{x_1,x_0}\left((H(x_1\mid z))\left(-\exp^{-1}_{x_1}(\gamma(t_z))\right)\right)\\
&=&(H(x_0\mid z))\left(\exp^{-1}_{x_0}(x_1)\right)\\
&&-\,\frac{\rho(\gamma(t_z),x_1)}{\rho(x_0,x_1)}\left\{(H(x_0\mid z))\left(\exp^{-1}_{x_0}(x_1)\right)\!+\!\Pi_{x_1,x_0}\left((H(x_1\mid z))\left(\exp^{-1}_{x_1}(x_0)\right)\right)\right\}\\
&=&(H(x_0\mid z))\left(\exp^{-1}_{x_0}(x_1)\right)+o(\rho(x_0,x_1)),
\end{eqnarray*}
so that the required result follows.
\end{proof}

In the remainder of the paper it will be useful to use the different but equivalent representation of $\mathcal{A}_{\delta}(x)$ in (\ref{curlyAdelta}) given by
\begin{equation}
\mathcal{A}_{\delta}(x)=\{ z \in \M: \, \mathcal{C}_z \cap B_\delta(x) \neq \emptyset\}.
\label{curlyAdelta2}
\end{equation}
To see that (\ref{curlyAdelta}) and (\ref{curlyAdelta2}) are equivalent, note that
\begin{align}
\mathcal{A}_{\delta}(x)&=\{ z \in \M: \, \mathcal{C}_z \cap B_\delta(x) \neq \emptyset\} \nonumber\\
&=\left \{ z \in \M: \, \textrm{there exists} \,  y \, \, \textrm{such that} \, \, y \in \mathcal{C}_z \, \, \& \, \, y \in B_\delta(x)   \right \} \nonumber \\
&=\left \{ z \in \M: \, \textrm{there exists}\,  \, y \,  \, \textrm{such that} \, \,  z \in \mathcal{C}_y \, \, \& \, \,  y \in B_\delta(x)   \right \} \nonumber \\
&= \bigcup_{y \in B_{\delta}(x)} \mathcal{C}_y,
\label{curlyAdelta3}
\end{align}
using the fact that $z \in \mathcal{C}_y$ if and only if $y \in \mathcal{C}_z$.

\vskip 0.2truein 
\noindent\textit{Proof of Theorem} 1. 
%\begin{proof}
When $x_0$ and $x_1$ are sufficiently close, write $\gamma$ for the unique unit speed geodesic from $x_0$ to $x_1$ and $\N^*_{x_0,x_1}$ for the set defined by
\begin{equation}
\N^*_{x_0,x_1}=\{z\in\M\mid\gamma(t)\in\H_z\hbox{ for some }t\in(0,\rho(x_0,x_1))\}.
\label{curly_N_star}
\end{equation}
It follows from condition (iv) of Theorem 1 that the volume of $\N^*_{x_0,x_1}$ is $O(\rho(x_0,x_1))$, because for $\delta>0$ sufficiently small and $x_1$ such that $\rho(x_0, x_1)=\delta$,
$\mathcal{N}_{x_0, x_1}^\ast \subseteq \mathcal{A}_\delta(x_0)$ and also $\mu(\mathcal{A}_\delta (x_0))=O(\delta)$ as $\delta \downarrow 0$.   Similar to $\N^*_{x_0,x_1}$, we also write $\mathcal A_{x_0,x_1}$ for the set defined by
\[\mathcal A_{x_0,x_1}=\{z\in\M\mid\gamma(t)\in\C_z\hbox{ for some }t\in(0,\rho(x_0,x_1))\}.\] 
Then $\mathcal A_{x_0,x_1}\supseteq\N^*_{x_0,x}$.

\vskip 6pt
Since $x_0$ is the Fr\'echet mean of $\mu$, $G_\mu(x_0)=0$. Under the given assumption, we also have that, for $x\in\M$ in a neighbourhood of $x_0$, $\mu(\{\xi\in\C_x\})=0$. Thus, it follows from %Lemma \ref{le3}, 
\eqref{eqn3.1} and Proposition \ref{prop2} that, for $x$ sufficiently close to $x_0$, 
\begin{eqnarray*}
&&\hspace*{-5mm}\Pi_{x,x_0}(G_\mu(x))\\
&=&\left(\int_{\M}H(x_0\mid\xi)\,\d\mu(\xi)\right)\left(\exp^{-1}_{x_0}(x)\right)\\
&&+\int_{\M}\rho_\xi(\gamma(t_\xi))\,\kappa(\gamma(t_\xi)\mid \xi)\,\Pi_{\gamma(t_\xi),x_0}\left(\n(\gamma(t_\xi)\mid \xi)\right)\,1_{\N^*_{x_0,x}}(\xi)\,\d\mu(\xi)\\
&&+\,\,\Pi_{x,x_0}\left(\int_{\mathcal A_{x_0,x}\setminus\N^*_{x_0,x}}\exp^{-1}_x(\xi)\,1_{\{\xi\not\in\C_x\}}\d\mu(\xi)\right)\\
&&-\left(\int_{\mathcal A_{x_0,x}\setminus\N^*_{x_0,x}}H(x_0\mid\xi)\,\d\mu(\xi)\right)\left(\exp^{-1}_{x_0}(x)\right)\!.
\end{eqnarray*}
Since condition (v) of Theorem together with Lemma \ref{le3} ensures that 
\[\int_{\mathcal A_{x_0,x}\setminus\N^*_{x_0,x}}H(x_0\mid\xi)\,\d\mu(\xi)=o(\rho(x_0,x)),\]
and since the boundedness of $\M$ and Lemma \ref{le3} together imply that
\[\int_{\mathcal A_{x_0,x}\setminus\N^*_{x_0,x}}\exp^{-1}_x(\xi)\,1_{\{\xi\not\in\C_x\}}\d\mu(\xi)=o(\rho(x_0,x)),\]
we have that 
\begin{eqnarray*}
&&\hspace*{-5mm}\Pi_{x,x_0}(G_\mu(x))\\
&=&\left(\int_{\M}H(x_0\mid\xi)\,\d\mu(\xi)\right)\left(\exp^{-1}_{x_0}(x)\right)\\
&&+\int_{\M}\rho_\xi(\gamma(t_\xi))\,\kappa(\gamma(t_\xi)\mid \xi)\,\Pi_{\gamma(t_\xi),x_0}\left(\n(\gamma(t_\xi)\mid \xi)\right)\,1_{\N^*_{x_0,x}}(\xi)\,\d\mu(\xi)\\
&&+\,\,o(\rho(x_0,x)).
\end{eqnarray*}
Thus, by the definition \eqref{eqn5c} of $\Psi_\mu (x_0)$, it is sufficient to show that
\begin{eqnarray}
\begin{array}{rcl}
&&\displaystyle\int_{\M}\rho_\xi(\gamma(t_\xi))\,\kappa(\gamma(t_\xi)\mid\xi)\,\Pi_{\gamma(t_\xi),x_0}(\n(\gamma(t_\xi)\mid\xi))\,1_{\N^*_{x_0,x}}(\xi)\,\d\mu(\xi)\\
&&\qquad=\left(J_\mu(x_0)\right)(\exp^{-1}_{x_0}(x))+o(\rho(x_0,x)),
\end{array}
\label{eqn8}
\end{eqnarray}
where $J_\mu(x_0)$ is defined by \eqref{eqn7}.

\vskip 6pt
For this, we note that the functions $\chi_i(\,\cdot\mid x)$ given in Corollary \ref{co1} are defined on a neighbourhood of $\H_x$. Thus, we may extend the definitions of the corresponding $\kappa(\,\cdot\mid x)$ and $\n(\,\cdot\mid x)$ given in \eqref{eqn1b} and \eqref{eqn1c} to that neighbourhood of $\H_x$. This implies that
\begin{eqnarray}
\begin{array}{rcl}
&&\hspace*{-5mm}\displaystyle\int_{\M}\rho_\xi(\gamma(t_\xi))\,\kappa(\gamma(t_\xi)\mid\xi)\,\Pi_{\gamma(t_\xi),x_0}(\n(\gamma(t_\xi)\mid\xi))\,1_{\N^*_{x_0,x}}(\xi)\,\d\mu(\xi)\\
&=&\displaystyle\int_{\M}\rho_\xi(x_0)\,\kappa(x_0\mid\xi)\,\n(x_0\mid\xi)\,1_{\N^*_{x_0,x}}(\xi)\,\d\mu(\xi)+O(\rho(x_0,x)^2).
\end{array}
\label{eqn3.3}
\end{eqnarray}

To analyse the right hand side of \eqref{eqn3.3} we consider, for any $z\in\N^*_{x_0,x_1}$, the minimal unit speed geodesic $\beta_z$ from $z$ to $x_0$. 
Extending $\beta_z$ backwards beyond $z$, let 
$y_z$ be the first hitting point of $\C_{x_0}$ on the extension; see Figure \ref{fig1}.
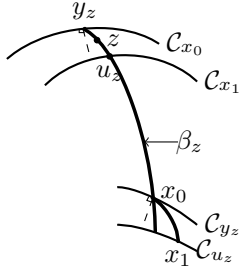
\begin{figure}
\begin{center}
\begin{tikzpicture}[scale=0.5]
\draw [line width=1.3pt] plot [smooth, tension=1.2] coordinates { (0.5,3) (1.8,1) (2.4,-2.4)};
\draw [line width=1.3pt] plot [smooth, tension=1.2] coordinates { (2.33,-1.48) (2.8,-2) (3,-2.63)};
\draw [line width=0.8pt] plot [smooth, tension=1.2] coordinates {(-1.5,2.2) (0.5,3) (2.5,2.6)};
\draw [line width=0.8pt] plot [smooth, tension=1.2] coordinates {(-0.5,1.5) (1.5,2.3) (3.5,1.6)};
\draw [line width=0.8pt] plot [smooth, tension=1.2] coordinates {(1.4,-1.2) (2.4,-1.5) (3.5,-2.2)};
\draw [line width=0.8pt] plot [smooth, tension=1.2] coordinates {(1.4,-2.2) (2.4,-2.4) (3.5,-2.9)};
\draw [dashed] (0.5,3) -- (0.7,2.2);
\draw [dashed] (2.33,-1.5) -- (2.05,-2.3);
\draw (0.4,3) -- (0.4,2.9) -- (0.54,2.9);
\draw (2.23,-1.45) -- (2.18,-1.55) -- (2.32,-1.6);
\fill (0.55,3) circle (1.5pt) node[above] {$y_z$};
\fill (1.2,2.27) circle (2.5pt) node[left, below] {$u_z$};
\fill (0.87,2.7) circle (2.5pt) node[right] {$z$};
\fill (2.33,-1.48) circle (1.5pt);
\node at (2.9,-1.4) {$x_0$};
\fill (3,-2.63) circle (1.5pt) node[left, below] {$x_1$};
%\fill (2.4,-2.4) circle (1.5pt) node[below] {$w$};
\node at (3.2,2.6) {$\C_{x_0}$};
\node at (4.2, 1.6) {$\C_{x_1}$};
\node at (4.2,-2.2) {$\C_{y_z}$};
\node at (4,-2.9) {$\C_{u_z}$};
\draw [<-] (2.1,0) -- (3,0);
\node at (3.3,0) {$\beta_z$};
\end{tikzpicture}
\end{center}
\caption{Relevant points on the minimal geodesic $\beta_z$ from $x_0$ to $z$ when $z\not\in\C_{x_0}$}
\label{fig1}
\end{figure}
Let
\[\N_{x_0,x_1}=\{z\in\M\mid\gamma(t)\in\H_z\hbox{ for some }t\in(0,\rho(x_0,x_1))\hbox{ and }y_z\in\H_{x_0}\}.\]
Then, $\N_{x_0,x_1}$ is a Lebesgue measurable subset of $\N^*_{x_0,x_1}$ and the difference between the volumes of $\N_{x_0,x_1}$ and of $\N^*_{x_0,x_1}$ is $o(\rho(x_0,x_1))$.
Since, by (\ref{metric_stability}) and condition (iii) of Theorem 1, which states that  in a neighbourhood of $\C_{x_0}$, $\mu$ is absolutely continuous with respect to the volume measure ${\rm vol}_{\M}(\cdot)$ with continuous Radon-Nikodym derivative $\psi$, \eqref{eqn3.3} can be expressed in terms of $\N_{x_0,x}$ as
\begin{eqnarray}
\begin{array}{rcl}
&&\hspace*{-5mm}\displaystyle\int_{\M}\rho_\xi(\gamma(t_\xi))\,\kappa(\gamma(t_\xi)\mid\xi)\,\Pi_{\gamma(t_\xi),x_0}(\n(\gamma(t_\xi)\mid\xi))\,1_{\N^*_{x_0,x}}(\xi)\,\d\mu(\xi)\\
&=&\displaystyle\int_{\M}\rho_\xi(x_0)\,\kappa(x_0\mid\xi)\,\n(x_0\mid\xi)\,1_{\N_{x_0,x}}(\xi)\,\d\mu(\xi)+O(\rho(x_0,x)^2)\\
&=&\displaystyle\int_{\N_{x_0,x}}\!\!\!\!\!\!\rho_\xi(x_0)\,\kappa(x_0\mid\xi)\,\n(x_0\mid\xi)\,\psi(\xi)\,\d\hbox{vol}(\xi)+O(\rho(x_0,x)^2),
\end{array}
\label{eqn3.4}
\end{eqnarray}
where $x$ is sufficiently close to $x_0$. If we write $u_z$ for the point on $\beta_z$ that lies in $\C_{x_1}$ as in Figure \ref{fig1}, then the volume of the local cross-sectional slice of $\N_{x_0,x_1}$ at $z\in\N_{x_0,x_1}$ can be approximated by $\d\v_{\H_{x_0}}(y_z)\,\langle\n(y_z\mid x_0),\,\exp^{-1}_{y_z}(u_z)\rangle$ and, since $y_z\in\H_{x_0}$, we also have
\begin{eqnarray*}
\frac{\langle\n(y_z\mid x_0),\,\exp^{-1}_{y_z}(u_z)\rangle}{\langle\n(x_0\mid y_z),\,\exp^{-1}_{x_0}(x_1)\rangle}\approx\tau'_{y_z|x_0}(0),
%\label{eqn3.6}
\end{eqnarray*}
where both $\n(y_z\mid x_0)$ and $\n(x_0\mid y_z)$ are chosen such that the inner products are non-negative and where $\tau'_{y|x}(0)$ is given in \eqref{eqn5b}. These two facts together imply that, for $z\in\N_{x_0,x_1}$, 
\[\d\hbox{vol}(z)\approx\langle\n(x_0\mid y_z),\,\exp^{-1}_{x_0}(x_1)\rangle\,\tau'_{y_z|x_0}(0)\,\d\v_{\H_{x_0}}(y_z).\]
Using this and the continuity in $z$ of $\rho_z(x)$, $\kappa(x\mid z)$, $\n(x\mid z)$ and $\psi(z)$, the dominant term on the right hand side of the second equality in \eqref{eqn3.4} can be expressed as
\begin{eqnarray*}
&&\hspace*{-5mm}\int_{\N_{x_0,x}}\rho_\xi(x_0)\,\kappa(x_0\mid\xi)\,\n(x_0\mid\xi)\,\psi(\xi)\,\d\hbox{vol}(\xi)\\
&=&\int_{\H_{x_0}}\!\!\!\!\!\rho_y(x_0)\,\kappa(x_0\mid y)\,\n(x_0\mid y)\,\langle\n(x_0\mid y),\,\exp^{-1}_{x_0}(x)\rangle\,\tau'_{y|x_0}(0)\,\psi(y)\,\d\hbox{vol}_{\H_{x_0}}(y)\\
&&+\,\,o(\rho(x_0,x)).
\end{eqnarray*}
Hence, \eqref{eqn8} follows from the definition \eqref{eqn7} of $J_\mu(x_0)$ as required.  \hfill $\square$

\section{Proof of Proposition 1 and Theorem 2}

\subsection{Proof of Proposition 1}   For each $n \geq 1$, let $\hat{\xi}_n \in \mathcal{G}_n$  denote any measureable selection from $\mathcal{G}_n$.  From the strong law of large numbers in Ziezold (1977), and using the assumption that $x_0$ is the unique population mean, almost surely
\[
\bigcap_{n=1}^\infty \overline{\bigcup_{k=n}^\infty \mathcal{G}_k}\subseteq  \{x_0\},
\]
where a horizontal line over a set indicates set closure.  From elementary considerations, the first set inclusion below holds and therefore the set $\mathcal{G}_0$ of limit points is
\[
\mathcal{G}_0=\bigcap_{n=1}^\infty \overline{\bigcup_{k=n}^\infty \{ \hat{\xi}_k\}}\subseteq \bigcap_{n=1}^\infty \overline{\bigcup_{k=n}^\infty \mathcal{G}_k}\subseteq  \{x_0\},
\]
where, for each $k\geq 1$, $\hat{\xi}_k \in \mathcal{G}_k$.  Since $\mathcal{G}_0 \subseteq \{x_0\}$, there are two possibilities: either $\mathcal{G}_0=\{x_0\}$, in which case the proposition follows; or, alternatively, $\mathcal{G}_0=\emptyset$, the empty set.  However, $\M$ is compact, so $\{\hat{\xi}_n\}$ must have a convergent subsequence with a limit $x_1 \in \M$.  Moreover, we must have $x_1 \in \mathcal{G}_0$ because $x_1$ is an accumulation point of the sequence.  Therefore, $x_1=x_0$, and consequently $\rho(\hat{\xi}_n, x_0) \rightarrow 0$ almost surely as required.  \hfill $\square$

\subsection{An elementary lemma}

We first introduce some notation.  If  $X=(X_1, \ldots , X_m)^\top$ and $x=(x_1, \ldots , x_m)^\top$ are vectors with real components then statements such as $\{X \leq x\}$ and $\{\vert X\vert \leq x\}$ are interpreted component-wise as $\{X_1 \leq x_1, \ldots , X_m \leq x_m\}$
and $\{\vert X_1\vert \leq x_1, \ldots , \vert X_m \vert \leq x_m\}$, respectively.  Also, denote by $\Phi(x; \Sigma)$ the cumulative distribution function of a zero-mean multivariate Gaussian distribution with covariance matrix $\Sigma$.  The Euclidean norm, $(w^\top w)^{1/2}$, of a vector $w \in \mathbb{R}^m$ is denoted $\vert \vert w \vert \vert$.  The following lemma is proved in Appendix B.

\begin{lemma}
Let $X, Y$ denote  $\mathbb{R}^m$-valued random vectors defined on an arbitrary  probability space $(\Omega, \mathcal{F}, \mathbb{P})$.  Let  $w \in \mathbb{R}^m$ denote a non-random vector with positive components.  Then
\begin{equation}
\vert  \mathbb{P}[X+Y \leq x]  - \mathbb{P}[X \leq x]\vert  
\leq \mathbb{P}[X \leq x+w]-\mathbb{P}[X\leq x-w] +2 \mathbb{P}[\{\vert Y \vert \leq w\}^c].
\label{first_bit}
\end{equation}
Moreover, suppose that for some $\epsilon>0$ and some $m \times m$  covariance matrix $\Sigma$, 
\begin{equation}
\sup_{x \in \mathbb{R}^m} \vert \mathbb{P}[X \leq x] - \Phi(x;\Sigma) \vert \leq \epsilon.
\label{asymptoticnormal}
\end{equation}
Then, for a constant $c_0>0$ depending only on $\Sigma$,
\begin{equation}
\vert  \mathbb{P}[X+Y \leq x] - \Phi(x;\Sigma)\vert 
\leq 3 \epsilon + 2 c_0 \vert \vert w \vert \vert  + 2 \mathbb{P}[\{\vert Y \vert \leq w\}^c].
\label{assymnorm2}
\end{equation}
\end{lemma}

Our proof of Theorem 2, in particular Step 1, makes use of this lemma. 

\subsection{Proof of Theorem 2}

The proof of Theorem 2 is broken into two steps.  In the first step we explain how Lemma 5 will be applied.  In the subsequent step, we explain how to make the right-hand side of (\ref{assymnorm2}) arbitrarily small uniformly for all $x \in \mathbb{R}^m$ and therefore the CLT in Theorem 2 will have been proved.

\vskip 0.2truein
\noindent \textbf{Step 1}.  \textit{Application of Lemma 5}.

\vskip 0.1truein
Write
\begin{equation}
G_{\hat{\mu}_n}(x)=\frac{1}{n}\sum_{i=1}^n \exp_{x}^{-1}(\xi_i)1_{\{\xi_i \notin \mathcal{C}_x\}}=\int_{\M}\exp_{x}^{-1}(\xi)1_{\{\xi \notin \mathcal{C}_x\}}\d\hat{\mu}_n(\xi),
\label{empirical_123}
\end{equation}
where $\hat{\mu}_n$ is the empirical distribution function on $\M$ based on the random sample $\xi_1, \ldots , \xi_n$ and define the vector field $\{Z_n(x): \, x \in \M\}$ on $\M$ by
\begin{equation}
Z_n(x)=\sqrt{n}\left [G_{\hat{\mu}_n}(x) - G_\mu(x)  \right],
\label{Z_vecfield}
\end{equation}
where $G_\mu (x)$ is defined in (\ref{G(x)}).  Under the conditions of Theorem 1 and 2, the population  Fr\'echet mean $x_0$ is a stationary minimum of (\ref{F_mu}) and, in particular,
\begin{equation}
G_\mu(x_0)= \int_{\M} \exp_{x_0}^{-1}(\xi) 1_{\{\xi \notin \mathcal{C}_{x_0}\}}\d\mu(\xi) = 0, 
\label{stationary_pop_mean}
\end{equation}
i.e. the zero element in $\mathcal{T}_{x_0}(\M)$, which follows from integrating (\ref{grad_d2}) over $\M$ with respect to the probability measure  $\mu$ and putting $x=x_0$.  Hence
\[
Z_n(x_0)=\sqrt{n}G_{\hat{\mu}_n}(x_0).
\]
Denote the Euclidean norm (which is the induced Riemannian tangent space  norm) on $\mathcal{T}_{x_0}(\M)$ by $\vert \vert \cdot \vert \vert$.  Since $\vert \vert \exp_{x_{0}}^{-1}(\xi)1_{\{\xi \notin \mathcal{C}_{x_0}\}}\vert \vert $ is bounded over $\xi \in M$, the LHS of (\ref{Z_vecfield}) with $x=x_0$ follows a central limit theorem in the tangent space, i.e. 
\begin{equation}
Z_n(x_0)=\frac{1}{\sqrt{n}}\sum_{i=1}^n \exp_{x_0}^{-1}(\xi_i)1_{\{\xi \notin \mathcal{C}_{x_0}\}} \overset{d} \to \mathfrak{N}_m(0_m,V_0),
\label{linear_CLT}
\end{equation}
where $V_0=\textrm{Cov}\{\exp_{x_0}^{-1}(\xi_1)1_{\{\xi \notin \mathcal{C}_{x_0}\}}\}=\textrm{Cov}\{\exp_{x_0}^{-1}(\xi_1)\}$.  

Moreover,  $\hat{\xi}_n \in \mathcal{G}_n$, which is assumed to be a measureable selection from $\mathcal{G}_n$ as in Proposition 1,  satisfies (\ref{stardust}) and consequently,
\begin{align}
Z_n(\hat{\xi}_n)&=\sqrt{n} [G_{\hat{\mu}_n}(\hat{\xi}_n)-G_\mu(\hat{\xi}_n)] \nonumber\\
&=- \sqrt{n}G_{\mu}(\hat{\xi}_n).
\label{other_point}
\end{align}

Define 
\begin{equation}
T_1=\Pi_{\hat{\xi}_n, x_0}Z_n(\hat{\xi}_n) - Z_n(x_0)
\label{def_T_1}
\end{equation}
Then, using (\ref{def_T_1}), (\ref{other_point}) and Theorem 1,  it is seen that
\begin{align}
T_1&=\Pi_{\hat{\xi}_n, x_0} Z_n(\hat{\xi}_n)-Z_n(x_0)\nonumber \\
&=-\sqrt{n} \Pi_{\hat{\xi}_n, x_0}G_\mu(\hat{\xi}_n)-Z_n(x_0) \nonumber \\
&=\sqrt{n} \Psi_\mu(x_0)\exp_{x_0}^{-1}(\hat{\xi}_n)-\sqrt{n}R(\hat{\xi}_n, x_0) - Z_n(x_0).
\label{ensign1}
\end{align}
Since, by assumption (vi) of Theorem 2, $\Psi_\mu(x_0)$ has full rank, it follows that
\begin{align}
\sqrt{n}\exp_{x_0}^{-1}(\hat{\xi}_n)&=\Psi_\mu(x_0)^{-1}Z_n(x_0) +\Psi_\mu(x_0)^{-1} \left \{T_1 + \sqrt{n}R(\hat{\xi}_n, x_0)\right \} \nonumber \\
&=X+Y,
\label{ensign2}
\end{align}
where
\begin{equation}
X=\Psi_\mu(x_0)^{-1}Z_n(x_0)
\label{ensign3}
\end{equation}
and
\begin{equation}
Y=\Psi_\mu(x_0)^{-1} \left \{T_1 + \sqrt{n}R(\hat{\xi}_n, x_0)\right \}\!.
\label{ensign4}
\end{equation}

To establish Theorem 2, we apply Lemma 5 with $X$ and $Y$ defined in (\ref{ensign3}) and (\ref{ensign4}), respectively.  Since, from (\ref{linear_CLT}), we know that 
$Z_n(x_0) \overset{d} \rightarrow \mathfrak{N}_m(0_m, V_0)$, it follows that $\Psi_\mu(x_0)^{-1}Z_n(x_0)$ is asymptotically normal with mean vector the zero vector and covariance matrix $\Psi_\mu(x_0)^{-1} V_0\left\{\Psi_\mu(x_0)^\top\right \}^{-1}$.
Moreover, as $n \rightarrow \infty$,  a suitable sequence of $w$'s such that $\vert \vert w \vert \vert \rightarrow 0$ and $\mathbb{P}[\{\vert Y \vert \leq w\}^c] \rightarrow 0$ can always be found provided all components of $Y$ go to $0$ in probability.  Consequently, to complete the proof of Theorem 2,  it is sufficient to show that $\vert \vert Y \vert \vert \overset{p} \rightarrow 0$, which is proved in Step 2.

\vskip 0.2truein
\noindent \textbf{Step 2}.  \textit{Show that $\vert \vert Y \vert \vert \overset{p} \rightarrow 0$.}

\vskip 0.1truein
In Step 2 we first show that $\vert \vert T_1 \vert \vert \overset{p} \rightarrow 0$, where $T_1$ is defined in (\ref{ensign1}).  Then we deduce that $\vert \vert Y \vert \vert \overset{p} \rightarrow 0$, where $Y$ is defined in (\ref{ensign4}).  To establish the result for $T_1$  we shall make use of results from empirical process theory.
A key step is to approximate
\begin{equation}
\mathbb{E}\left [ \textrm{tr} \left \{\left (\Pi_{x,x_0}Z_n(x) - \Pi_{y,x_0} Z_n(y) \right )^\top  \left ( \Pi_{x,x_0}Z_n(x)- \Pi_{y,x_0}Z_n(y)  \right )   \right \} \right ]\!.
\label{SD_norm_4}
\end{equation}
However, as $\Pi_{x,x_0}$ and $\Pi_{y,x_0}$ are vector space  isomorphisms from $\mathcal{T}_x(\M)$ to $\mathcal{T}_{x_0}(\M)$ and $\mathcal{T}_y(\M)$ to $\mathcal{T}_{x_0}(\M)$, respectively, it follows from the definition of $Z_n(x)$ in (\ref{Z_vecfield}) that $\Pi_{x,x_0} Z_n(x) - \Pi_{y,x_0}Z_n(y)$ in (\ref{SD_norm_4}) is an IID sum of terms $q_x(\xi_i)-q_y(\xi_i)$, where
\[
q_x(\xi_i) =\Pi_{x,x_0} \exp_x^{-1} (\xi_i)1_{\{\xi_i \notin \mathcal{C}_x\}} - \Pi_{x,x_0}G_\mu(x),
\]
for $i=1, \ldots , n$, with a similar definition for $q_y(\xi_i)$,  and with $q_x(\xi_i), q_y(\xi_i)  \in \mathcal{T}_{x_0}(\M)$.  It follows that (\ref{SD_norm_4}) is equal to
\begin{equation}
\int_{\M} \left \{  q_x(\xi)-q_y(\xi)   \right \}^\top \left \{q_x (\xi) - q_y(\xi)\right \} \d\mu(\xi).
\label{SD_norm_5}
\end{equation}
  It is also assumed below that $x,y  \in B_{\delta_0}(x_0)$, the open ball in $\M$ of radius $\delta_0$ centred at $x_0$, where $0 <\delta_0 <\delta$, and using condition (iii) of Theorem 1, $\delta$ has been chosen to be sufficiently small for $\mu$ to be absolutely continuous on $\mathcal{A}_\delta (x_0)$, where $\mathcal{A}_\delta (x_0)$ is defined in (\ref{curlyAdelta}) and $x_0 \in \M$ is the population Fr\'echet mean of $\mu$, assumed to be unique. 

We may write (\ref{SD_norm_5}) as
\[
\int_{\M}  = \int_{\mathcal{A}_\delta(x_0)} + \int_{\M \setminus \mathcal{A}_\delta(x_0)}.
\]
For $\xi  \in \M \setminus \mathcal{A}_\delta(x_0)$, using (\ref{eqn3.1}) and Theorem 1, we have
\begin{align*}
q_x(\xi) - q_y(\xi)& = \Pi_{x,x_0}\{\exp_x^{-1}(\xi)1_{\{\xi \notin \mathcal{C}_x\}} - G_\mu(x)\} \\
& \hskip 0.3truein -\Pi_{y,x_0}\{\exp_y^{-1}(\xi)1_{\{\xi \notin \mathcal{C}_y\}}-G_\mu(\xi)\} \\
& =H(x_0\vert \xi) \{exp_{x_0}^{-1}x - \exp_{x_0}^{-1}(y)\} + R_1(x_0,x) - R_1(x_0,y)\\
& \hskip 0.2truein - \Psi_\mu(x_0) \{\exp_{x_0}^{-1}(x) - \exp_{x_0}^{-1}(y)\} +R(x,x_0)- R(y,x_0);
\end{align*}
and also, since $\M$ is smooth and compact, it follows that
\[
\vert \vert R(x,x_0) - R(y,x_0)\vert \vert = O(\rho(x,y))=\vert \vert R_1(x_0,x)-R_1(x_0,y)\vert \vert.
\]
Consequently, for $\xi \in\M \setminus  \mathcal{A}_\delta(x_0)$,
\begin{equation}
\vert \vert q_x(\xi) - q_y(\xi) \vert \vert = O(\rho(x,y)).
\label{SD_norm_17}
\end{equation}
Therefore, since (\ref{SD_norm_17}) holds uniformly for $z \in \M \setminus \mathcal{A}_\delta(x_0)$ by compactness, connectedness and smoothness of $\M$, it follows that
\begin{equation}
\int_{\M \setminus \mathcal{A}_\delta(x_0)} \{q_x(\xi) - q_y(\xi)\}^\top \{q_x(\xi) - q_y(\xi)  \} \d\mu(\xi) = O(\rho(x,y)^2).
\label{SD_norm_18}
\end{equation}

To approximate the integral of $\{q_x(\xi) -q_y(\xi)\}^\top\{ q_x(\xi)-q_y(\xi) \}$ on the set $\mathcal{A}_{\delta}(x_0)$, we use the following facts: recall that $\delta>0$ has been chosen sufficiently small so that the Radon-Nickodym derivative of $\mu$ has a continuous version on $\mathcal{A}_{\delta}(x_0)$ (see assumption (iii) of the theorems);  the Riemannian volume of $\mathcal{A}_{\delta}(x_0)$ satisfies
 $\textrm{vol}_{\M}(\mathcal{A}_{\delta}(x_0))=O(\rho(x,y))$ (see  immediately below (\ref{curly_N_star})); and  $q_x(\xi)$ is bounded on $\M$.  As a consequence of these facts, 
\begin{equation}
\int_{\mathcal{A}_{\delta}(x_0)}  \textrm{tr}\left  [ \{ q_x(\xi) - q_y(\xi) \}^\top \{q_x(\xi) - q_y(\xi)\} \right ]\d\mu(\xi) =O(\rho(x,y)).
\end{equation}
Consequently, for $x$ and $y$ such that  $\rho(x_0,x) \rightarrow 0$ and $\rho(x_0,y) \rightarrow 0$,
\begin{align}
&\sqrt{\int_{\M} \textrm{tr}\left [ \left \{  q_x(\xi)-q_y(\xi)   \right \}^\top \left \{q_x (\xi) - q_y(\xi)\right \} \right ] \d\mu(\xi)}
= \sqrt{\int_{\mathcal{A}_{\delta}(x_0)} + \int_{\M \setminus \mathcal{A}_\delta(x_0)}  }\nonumber \\
& \hskip 1.0truein = \sqrt{O(\rho(x,y)) +O( \rho(x,y)^2)} \nonumber \\
&\hskip 1.0truein =O \{\rho(x,y)^{1/2} \}.
\label{SD_norm_20}
\end{align}

The relevant class of functions here is
\[
\mathcal{F}_{\delta_0} =\{q_x(\cdot): \, q_x: \M \rightarrow \mathbb{R}^m , \, \rho(x,x_0)<\delta_0\},
\]
where $0<\delta_0<\delta$ is chosen to be sufficiently small.  Using (\ref{SD_norm_20}) and the fact that $\M$ is compact, it follows  that the integrals in Theorem 2.5.6 of van der Vaart and Wellner (1996) are both finite, so that $\mathcal{F}_{\delta_0}$ is a \textit{Donsker} class.  By Theorem 3.34 of Dudley (2008), the Donsker property is sufficient to guarantee asymptotic equicontinuity, which in turn implies that 
\begin{equation}
\vert \vert \Pi_{\hat{\xi}_n,x_0}Z_n(\hat{\xi}_n)-Z_n(x_0)\vert \vert = o_p(1). 
\label{SD_norm_19}
\end{equation}
Thus we have proved that $\vert \vert T_1 \vert \vert \overset{p} \rightarrow 0$.
 One further comment: most of the results in the  literature on empirical process theory, including  van der Vaart and Wellner (1996) and Dudley (2008), are usually stated for classes of real-valued functions. To generalise to $\mathbb{R}^m$-valued functions, where $m$ is finite, is a straightforward matter.  In the present context, we simply prove that each component is $o_p(1)$, which follows immediately from the calculations given above.

We now complete the proof of Step 2.  Recall the first equality in (\ref{ensign2}).  Since, from Theorem 1,
\[
R(\hat{\xi}_n, x_0) = o(\rho(\hat{\xi}_n, x_0))=o\left ( \vert \vert \exp_{x_0}^{-1}(\hat{\xi}_n)\vert \vert \right)\!,
\]

it follows that
\[
\sqrt{n}R(\hat{\xi}_n,x_0)=o\left (\sqrt{n} \vert \vert \exp_{x_0}^{-1}(\hat{\xi}_n)\vert \vert \right)\!.
\]
Moreover, it has already been shown that $\vert \vert T_1 \vert \vert \overset{p}  \rightarrow 0$, and we know from condition (v) of Theorem 2 that $\Psi_{\mu}(x_0)^{-1}$ is a fixed matrix with bounded elements and that 
$\vert \vert \Psi_\mu(x_0)^{-1}Z_n(x_0)\vert \vert =O_p(1)$ due to the central limit theorem.  Consequently, it must be the case that $\sqrt{n} \vert \vert \exp_{x_0}^{-1}(\hat{\xi}_n)\vert \vert = O_p(1)$.  Hence $\sqrt{n}\vert \vert R(\hat{\xi}_n,x_0)\vert \vert =o_p(1)$
and therefore $\vert \vert Y \vert \vert =o_p(1)$ as claimed.  \hfill $\square$

%\begin{appendix}
\section*{Appendix A: Proof of Proposition 2}

For a Riemannian manifold $(\M,g)$ with connection $\nabla$, the Sasaki metric $\hat g$ on the tangent bundle $\T\!\!\M$ is a natural Riemannian metric that has the properties that (i) horizontal and vertical distributions are orthogonal; (ii) the metric induced on the fibers is Euclidean; (iii) the canonical projection from $(\T\!\!\M,\hat g)$ to $(\M,g)$ a Riemannian submersion. More precisely, it is determined by
\[\hat g_{(x,u)}(X^h,Y^h)=g_x(X,Y),\quad \hat g_{(x,u)}(X^v,Y^h)=0,\quad \hat g_{(x,u)}(X^v,Y^v)=g_x(X,Y)\]
for all vector fields $X,Y\in\mathcal{C}^\infty(\T\!\!\M)$, where $X^h$ and $X^v$ respectively denote the horizontal and vertical lifts of $X$ and $Y$ to $\T(\T\!\!\M)$. A smooth curve $\Gamma(t)=(\gamma(t),u(t))$ in $\T\!\!\M$ is a geodesic under the Sasaki metric if and only if it satisfies 
\begin{eqnarray}
	\nabla_{\dot\gamma(t)}\dot\gamma(t)+R\left(u(t),\nabla_{\dot\gamma(t)}u(t)\right)\,\dot\gamma(t)=0
	\label{eqn1}
\end{eqnarray}
and 
\begin{eqnarray}
	\nabla_{\dot\gamma(t)}\nabla_{\dot\gamma(t)}u(t)=0.
	\label{eqn2a}
\end{eqnarray}
In particular, if $\gamma(t)$ is a geodesic on $\M$ and $u(t)$ is a parallel vector field along $\gamma(t)$, then $\Gamma(t)=(\gamma(t),u(t))$ satisfies the conditions \eqref{eqn1} and \eqref{eqn2a}, so that $\Gamma(t)$ is a geodesic on $\T\!\!\M$ equipped with the Sasaki metric. Such a $\Gamma(t)$ is called a horizontal lift of $\gamma(t)$. 

\vskip 6pt
\noindent {\it Proof of Proposition 2}.
	It is sufficient to show that, for any $\delta>0$, there is $\delta_1>0$ such that, if $\rho(x,y)<\delta_1$, then for any $y'\in{\C}_y$, there is an $x'\in{\C}_x$ such that $\rho(x',y')<\delta$.
	
For this, we note first that the `full' exponential map $\hbox{Exp}\!:\T\!\!\M\longrightarrow\M$ is $\mathcal C^\infty$ and that, since $\M$ is compact, the distance $\tilde\rho(x,v)$ of $x$ to the cut point of $x$ along the geodesic $\exp_x(tv)$ is a continuous function to $\mathbb{R}^+$ on the unit sphere bundle $\S\!\M = \{(x,v)\in\T\!\!\M\mid\|v\|=1\}$ in the tangent bundle. Thus, it follows from the compactness of $\M$ that the function $R$ defined by
	\[R: \S\!\M\rightarrow\M;\,v_x=(x,v)\mapsto\exp_x(\tilde\rho(x,v)v)\]
	is uniformly continuous. Hence, for any $\delta>0$, there is $\delta_1>0$ such that, for any $v_x=(x,v),\,w_y=(y,w)\in\S\!\M$, $\hat\rho(v_x,\,w_y)<\delta_1$ implies that
	\begin{eqnarray}
		\rho(R(v_x),\,R(w_y))<\delta,
		\label{eqn3a}
	\end{eqnarray}
where $\hat\rho$ is the distance function on $\T\!\!\M$ induced by the Riemannian metric $\hat g$.
	
	\vskip 6pt
	Now, for given $x\in\M$, assume $\rho(x,y)<\delta_1$. For any $y'$ in $\C_y$, let $w\in \S_y(\M)$ such that $y'=R(w_y)$. For such a $w_y=(y,w)\in\S\!\M$, take $v\in\S_x(\M)$ to be the parallel transport of $w$ to $x$ along the geodesic from the $y$ to $x$. Then, $v_x=(x,v)\in\S\!\M$. By the remark following \eqref{eqn2}, the distance $\hat\rho(v_x,w_y)$ between $v_x$ and $w_y$ on $\T\!\!\M$ is equal to $\rho(x,y)<\delta_1$ on $\M$. Take $x'=R(v_x)$. Then, $x'\in\C_x$ and it follows from \eqref{eqn3a} that $\rho(x',y')=\rho(R(v_x),\,R(w_y))<\delta$ as required.  \hfill $\square$

\section*{Appendix B: Proof of Lemma 5}

We first prove (\ref{first_bit}).  For any  events $A$ and $B$ we have
\begin{equation}
\mathbb{P}[A]-\mathbb{P}[B^c] \leq \mathbb{P}[A \cap B] \leq \mathbb{P}[A].
\label{A.1}
\end{equation}
Also, if $w$ has positive components, $A_0=\{X+Y \leq x\}$, $A_1=\{X \leq x+w\}$, $A_2=\{X \leq x-w\}$, $A_3=\{X \leq x\}$ and $B=\{\vert Y \vert \leq w\}$, then
\[
A_2 \cap B \subseteq A_0 \cap B \subseteq A_1 \cap B,
\]
and consequently,
\begin{equation}
\mathbb{P}[A_2 \cap B] \leq \mathbb{P}[A_0 \cap B] \leq \mathbb{P}[A_1 \cap B].
\label{A.2}
\end{equation}
Applying the LHS  ineqality of (\ref{A.1}) to the LHS of (\ref{A.2})  and the RHS inequality of (\ref{A.1}) to the RHS of (\ref{A.2}), we obtain
\[
\mathbb{P}[A_2] - \mathbb{P}[B^c] \leq \mathbb{P}[A_0 \cap B] \leq \mathbb{P}[A_1].
\]
From this it follows that
\[
\mathbb{P}[A_3] - \mathbb{P}[A_2]+\mathbb{P}[B^c] \geq \mathbb{P}[A_3]-\mathbb{P}[A_0 \cap B]
\]
and 
\[
\mathbb{P}[A_0 \cap B] - \mathbb{P}[A_3]  \leq \mathbb{P}[A_1] - \mathbb{P}[A_3].
\]
Consequently,
\begin{align*}
\vert  \mathbb{P}[A_0 \cap B] - \mathbb{P}[A_3]   \vert & \leq \mathbb{P}[A_1]- \mathbb{P}[A_3] + \mathbb{P}[A_3]- \mathbb{P}[A_2] + \mathbb{P}[B^c]\\
&= \mathbb{P}[A_1]- \mathbb{P}[A_2] + \mathbb{P}[B^c].
\end{align*}
Moreover, since $\mathbb{P}[A_0 \cap B]=\mathbb{P}[A_0] - \mathbb{P}[A_0 \cap B^c]$ and for any real numbers $a$ and $b$, $\vert a-b\vert \geq \vert a\vert - \vert b \vert$, it follows that
\begin{align*}
\vert  \mathbb{P}[A_0 \cap B] - \mathbb{P}[A_3]   \vert &=\vert \mathbb{P}[A_0] - \mathbb{P}[A_3]- \mathbb{P}[A_0\cap B^c] \vert\\
&\geq \vert \mathbb{P}[A_0] - \mathbb{P}[A_3] \vert - \mathbb{P}[A_0 \cap B^c].
\end{align*}
Therefore, since $\mathbb{P}[A_0 \cap B^c] \leq \mathbb{P}[B^c]$, it follows that
\[
\vert \mathbb{P}[A_0] - \mathbb{P}[A_3] \vert \leq  \mathbb{P}[A_1]- \mathbb{P}[A_2] +2 \mathbb{P}[B^c].
\]

The inequality in (\ref{assymnorm2}) follows easily from (\ref{asymptoticnormal}), because (\ref{asymptoticnormal}) implies that
\[
\vert \mathbb{P}[X \leq x] - \Phi(x;\Sigma)\vert \leq \epsilon \hskip 0.2truein \hbox{and}  \hskip 0.2truein \vert \mathbb{P}[X \leq x \pm w] - \Phi(x \pm w;\Sigma)\vert \leq \epsilon 
\]
for all $x \in \mathbb{R}^m$, and for all $w \in \mathbb{R}^m$ with positive components.  Moreover, using Taylor expansion it is easily shown that
\[
\vert \Phi (x \pm w; \Sigma) - \Phi(x; \Sigma) \vert  \leq c_0 \vert \vert w \vert \vert ,
\]
where $c_0>0$ is a suitable constant depending only on $\Sigma$.
The proof is now complete.  \hfill $\square$

%\end{appendix}

\end{document}